\documentclass[12pt]{amsart}

\usepackage{amsfonts, amsmath, amssymb, amscd, latexsym, graphicx, float,amsthm}

\newtheorem{thm}{Theorem}[section]
\newtheorem{prop}[thm]{Proposition}

\newtheorem{lem}[thm]{Lemma}
\newtheorem{cor}[thm]{Corollary}

\newtheorem{defn}[thm]{Definition}

\newtheorem*{thm1}{Theorem 1}
\newtheorem*{thm2}{Theorem 2}

\def\nn{\mathbb{N}}
\def\zz{\mathbb{Z}}

\def\rr{\mathbb{R}}

\newcommand{\st}{\colon}

\newcommand{\iso}{\cong}

\def\yy{y_\mathtt}
\def\xx{x_\mathtt}

\begin{document}

\title{Commutators in groups of piecewise projective homeomorphisms.}

\author{Jos\'e Burillo}
\address{Departament de Matem\`atiques, UPC, C/Jordi Girona 1-3,
08034 Barcelona, Spain} \email{burillo@ma4.upc.edu}

\author{Yash Lodha}
\address{Department of Mathematics, EPFL, Lausanne, Switzerland.}\email{yash.lodha@epfl.ch}

\author{Lawrence Reeves}
\address{School of Mathematics and Statistics, University~of~Melbourne, Victoria, Australia}
\email{lreeves@unimelb.edu.au}

\thanks{The first author's research is supported by MICINN grant MTM2014-54896-P.
The second author's research is supported by an EPFL-Marie Curie grant.
The second author would like to thank Swiss Air for its hospitality during a flight on which a portion of the paper was written.}

\begin{abstract}
In \cite{Monod} Monod introduced examples of groups of piecewise projective homeomorphisms which are not amenable and which do not contain free subgroups,
and in \cite{LodhaMoore} Lodha and Moore introduced examples of finitely presented groups with the same property.
In this article we examine the normal subgroup structure of these groups.
Two important cases of our results are the groups $H$ and $G_0$.
We show that the group $H$ of piecewise projective homeomorphisms of $\rr$ has the property that $H''$ is simple and that every
proper quotient of $H$ is metabelian.
We establish simplicity of the commutator subgroup of the group $G_0$,
which admits a presentation with $3$ generators and $9$ relations.
Further, we show that every proper quotient of $G_0$ is abelian.
It follows that the normal subgroups of these groups
are in bijective correspondence with those of the abelian (or metabelian) quotient.
\end{abstract}

\maketitle

\section*{Introduction}

In \cite{Monod} Monod proved that the group $H$ of piecewise projective homeomorphisms of the real line is non-amenable and does not contain non-abelian free subgroups.
This provides a new counterexample to the so called von Neumann--Day problem \cite{vN,Day}.
In fact, Monod introduced a family of groups $H(A)$ for a subring $A$ of $\mathbb{R}$.
In the case where $A$ is strictly larger than $\zz$, they were all demonstrated to be counterexamples.
The group
$H$ is the case in which $A=\mathbb{R}$.

The subgroups $G_0$ and $G$ of $H$ were introduced by Lodha and Moore in \cite{LodhaMoore} as finitely presented counterexamples.
The groups $G_0$ and $G$ share many features with Thompson's group $F$.
They can be viewed as groups of homeomorphisms of the Cantor set of infinite binary sequences,
and as groups of homeomorphisms of the real line.
They admit small finite presentations, and symmetric infinite presentations
with a natural normal form \cite{LodhaMoore,Lodha}.
Further, they are of type $F_{\infty}$ \cite{Lodha}.
Viewed as homeomorphisms of the Cantor set, the elements can be represented by tree diagrams.


Thompson's group $F$ satisfies the property that $F'$ is simple, and every proper quotient of $F$ is abelian \cite{cfp}.
In this article we examine the normal subgroup structure, and in particular the commutator subgroup structure of $G$, $G_{0}$, and $H(A)$ for a subring $A$ of
$\mathbb{R}$, and obtain properties similar to $F$.
We prove the following.

\newcommand{\mainA}{1}
\begin{thm1}
Let $A$ be a subring of $\mathbb{R}$.
If $A$ has units other than $\pm 1$, then:

\begin{enumerate}
\item $H(A)'\neq H(A)''$.
\item $H(A)''$ is simple.
\item Every proper quotient of $H(A)$ is metabelian.
\end{enumerate}

If the only units in $A$ are $\pm 1$, then:

\begin{enumerate}
\item $H(A)'$ is simple.
\item Every proper quotient of $H(A)$ is abelian.
\item All finite index subgroups of $H(A)$ are normal in $H(A)$.
\end{enumerate}
\end{thm1}

We show the following for the finitely presented groups $G_0$ and $G$ (defined in Section \ref{sec:background}).

\newcommand{\mainB}{2}
\begin{thm2}
The group $G_0$ satisfies the following:
\begin{enumerate}
\item $G_0^{\prime}$ is simple.
\item Every proper quotient of $G_0$ is abelian.
\item All finite index subgroups of $G_0$ are normal in $G_0$.
\end{enumerate}
The group $G$ satisfies the following:
\begin{enumerate}
\item $G''\neq G'$.
\item $G''$ is simple and $G''=G_0'$.
\item Every proper quotient of $G$ is metabelian.
\end{enumerate}
\end{thm2}

One interesting feature of this article is that although the proofs of Theorems \mainA\ and \mainB\ both use a theorem of Higman,
the proofs are different in the following sense.
The arguments of the proof of Theorem \mainB\ are intrinsic to the combinatorial model developed for $G_0,G$ in \cite{LodhaMoore} using continued fractions.
The arguments of the proof of Theorem \mainA\ arise in the setting of the action of the groups $H(A)$ on the real line.

\section{Background}\label{sec:background}

Actions considered will be from the right except when we
explicitly use function notation or represent elements by matrices.
The conjugate of $g$ by $h$ is $h^{-1}g h$.
All groups under study here are subgroups of $PPSL_2(\rr)$, the group of piecewise projective homeomorphisms of $\rr\cup\{\infty\}$ preserving orientation. Namely, for each element in $PPSL_2(\rr)$ there are finitely many points $t_1,\ldots,t_n$ such that in each of the intervals $(-\infty,t_1]$, $[t_i,t_{i+1}]$ and $[t_n,\infty)$ the map is of the form $t\mapsto(at+b)/(ct+d)$ for some matrix
$$
\left(
\begin{array}{cc}
a&b\\
c&d
\end{array}
\right)
$$
with determinant one. The group $H$ is the subgroup of $PPSL_2(\rr)$ formed of those maps that stabilize infinity, and that hence give homeomorphisms of $\rr$. Observe that elements of $H$ have affine germs at $\pm\infty$, that is, in the interval $[t_n,\infty)$ the map is $(at+b)/d$ with $ad=1$, and similarly on the interval $(-\infty,t_1]$. Given a subring $A$ of $\rr$, we denote by $P_A$ the set of fixed points of hyperbolic elements of $PSL_2(A)$.
Then $H(A)$ is defined to be the subgroup of $H$ consisting of elements that are piecewise $PSL_2(A)$ with breakpoints in $P_A$. See \cite{Monod} for details on these groups.

The two Lodha--Moore groups \cite{LodhaMoore} are finitely presented subgroups of $H$ and will be denoted by $G$ and $G_0$. The group $G_0$ is the group of homeomorphisms of $\rr$ generated by the following three maps:
$$
a(t)=t+1\qquad b(t)=\left\{\begin{array}{ll}
\smallskip
t&\text{ if }t\leq 0\\
\medskip
\dfrac{t}{1-t}&\text{ if }0\leq t\leq \dfrac12\\
\medskip
\dfrac{3t-1}{t}&\text{ if }\dfrac12\leq t\leq 1\\
t+1&\text{ if }1\leq t\end{array}\right.
\qquad
c(t)=\left\{\begin{array}{ll}\smallskip\dfrac{2t}{t+1}&\text{ if }0\leq t\leq1\\
t&\text{ otherwise}\end{array}\right.
$$
As is done with Thompson's group $F$, elements will be given an interpretation in terms of maps of binary sequences and tree diagrams. A \emph{binary sequence} is a (finite or infinite) sequence of $\mathtt{0}$ and $\mathtt1$. The set of infinite binary sequences will be denoted by $2^\nn$, and the set of finite ones by $2^{<\nn}$. We will use $\mathtt{s}$ and $\mathtt{t}$ to denote finite binary sequences (with a distinctive font), and Greek letters $\xi$, $\zeta$ or $\eta$ for infinite ones. Two binary sequences can be concatenated as long as the first one is finite, such as $\mathtt{010}\xi$, $\mathtt{s01}$ or $\mathtt{s}\zeta$.

We define the binary sequence map:
$$
\begin{array}c
x:2^\nn\longrightarrow 2^\nn\\
\begin{array}l
(\mathtt{00}\xi)\cdot x=\mathtt0\xi\\
(\mathtt{01}\xi)\cdot x=\mathtt{10}\xi\\
(\mathtt{1}\xi)\cdot x=\mathtt{11}\xi
\end{array}
\end{array}
\qquad
$$
and also, recursively, the pair of mutually inverse maps
$$
\begin{array}c
y:2^\nn\longrightarrow 2^\nn\\
\begin{array}l
(\mathtt{00}\xi)\cdot y=\mathtt0(\xi\cdot y)\\
(\mathtt{01}\xi)\cdot y=\mathtt{10}(\xi\cdot y^{-1})\\
(\mathtt{1}\xi)\cdot y=\mathtt{11}(\xi\cdot y)
\end{array}
\end{array}
\qquad
\begin{array}c
y^{-1}:2^\nn\longrightarrow 2^\nn\\
\begin{array}l
(\mathtt{0}\xi)\cdot y^{-1}=\mathtt{00}(\xi\cdot y^{-1})\\
(\mathtt{10}\xi)\cdot y^{-1}=\mathtt{01}(\xi\cdot y)\\
(\mathtt{11}\xi)\cdot y^{-1}=\mathtt{1}(\xi\cdot y^{-1}).
\end{array}
\end{array}
$$
Each of these maps will give rise to a family of maps of binary sequences defined in the following way. Given a finite binary sequence $\mathtt{s}$, the map $\xx{s}$ is the identity except on the infinite sequences starting with $\mathtt{s}$, where it acts as $x$ on the tail. That is,
$$
(\mathtt{s}\xi)\cdot \xx{s}=\mathtt{s}(\xi\cdot x)
$$
and as the identity if the sequence does not start with $\mathtt{s}$. The maps $\yy{s}$ or $\yy{s}^{-1}$ are defined in an analogous way. The map $\xx{s}$ (and also similarly $\yy{s}$) also admits a partial action on the set $2^{<\nn}$ of finite binary sequences in the natural way: if a sequence extends $\mathtt{s00}$, $\mathtt{s01}$ or $\mathtt{s1}$, we have $(\mathtt{su})\cdot \xx{s}=\mathtt{s}(\mathtt{u}\cdot x)$, and $\xx{s}$ is the identity on those sequences which are incompatible with $\mathtt{s}$. The map $\xx{s}$ is not defined on the other sequences, namely, $\xx{s}$ is not defined on $\mathtt{s0}$ or on the prefixes of $\mathtt{s}$, in particular, observe that $x$ is not defined on the sequence $\mathtt 0$.

The reason for these definitions is that they represent elements of $G_0$ under an identification given by the following maps:
$$
\begin{array}c
\medskip\varphi:2^\nn\longrightarrow[0,\infty]\\
\begin{array}l
\medskip\varphi(\mathtt0\xi)=\dfrac{1}{1+\frac{1}{\varphi(\xi)}}\\
\varphi(\mathtt1\xi)=1+\varphi(\xi)
\end{array}
\end{array}
\qquad
\begin{array}c
\Phi:2^\nn\longrightarrow\rr\\
\begin{array}l
\Phi(\mathtt0\xi)=-\varphi(\tilde\xi)\\
\Phi(\mathtt1\xi)=\varphi(\xi)
\end{array}
\end{array}
$$
where $\tilde\xi$ is the sequence obtained from $\xi$ by replacing all symbols $\mathtt0$ by $\mathtt1$ and viceversa. Under these definitions, the maps $a$, $b$ and $c$ are represented by the binary sequence maps $x$, $x_{\mathtt1}$ and $y_{\mathtt{10}}$, respectively. We have the following result (Proposition 3.1 in \cite{LodhaMoore}):

\begin{prop} For all $\xi$ in $2^\nn$ we have
$$
a(\Phi(\xi))=\Phi(x(\xi))\qquad b(\Phi(\xi))=\Phi(x_{\mathtt1}(\xi))\qquad c(\Phi(\xi))=\Phi(y_{\mathtt{10}}(\xi))
$$
\end{prop}

For all details and proofs, see \cite{LodhaMoore}. Hence, we can consider that the group $G_0$ is the group of maps of $2^\nn$ generated by $x,x_{\mathtt1},y_{\mathtt{10}}$.


The group $G$ is defined to be the group generated by all $\xx{s}$ and $\yy{s}$. The group $G_0$ is generated by all $\xx{s}$ and by all those $\yy{s}$ where $\mathtt{s}$ is not constant, that is, $\mathtt{s}$ is not $\mathtt{0}^n$ nor $\mathtt{1}^n$. We have a series of relations which are satisfied by these generators:

\begin{enumerate}
\item $\xx{s}^2=\xx{s0}\xx{s}\xx{s1}$,
\item if $\mathtt{t}\cdot \xx{s}$ is defined, then $\xx{t}\xx{s}=\xx{s}x_{\mathtt{t}\cdot \xx{s}}$,
\item if $\mathtt{t}\cdot \xx{s}$ is defined, then $\yy{t}\xx{s}=\xx{s}y_{\mathtt{t}\cdot \xx{s}}$,
\item if $\mathtt{s}$ and $\mathtt{t}$ are incompatible, then $\yy{s}\yy{t}=\yy{t}\yy{s}$,
\item $\yy{s}=\xx{s}\yy{s0}\yy{s10}^{-1}\yy{s11}$.
\end{enumerate}

The relations (1) and (2) are part of a known presentation for Thompson's group $F$ given by Dehornoy in \cite{D}. The key relation for the Lodha--Moore groups is the relation (5), which represents algebraically the recursive definition of $y$ given above.

Finally, the relations satisfied by these generators, and in particular the Thompson's group relations,
allow us to 
obtain finite presentations.
The group $G_0$ is generated only by $x$, $\xx{1}$ and $y_{\mathtt{10}}$ with a set of 9 relations, whereas $G$ is generated by $x$, $\xx{1}$, $y_{\mathtt0}$, $y_{\mathtt1}$ and $y_{\mathtt{10}}$.


In \cite{LodhaMoore} it is shown that every element of $G$ can be written in a standard form.
Recall that $G$ is generated by the set $X\cup Y$ where $X=\{ \xx{s} \st \mathtt{s}\in 2^{<\nn}\}$ and
$Y= \{\yy{s} \st \mathtt{s}\in 2^{<\nn}\}$.
 A word over $X\cup Y$ is said to be in \emph{standard form} if it is of the form
 $$
 hy_{\mathtt{s}_1}^{a_1}\ldots y_{\mathtt{s}_n}^{a_n}
 $$
 where $h$ is a word over $X$, $\mathtt{s}_j$ is a prefix of $\mathtt{s}_i$ only if $j\ge i$,
and the $a_i$ are arbitrary, non-zero integers.
It is shown in \cite[Lemma 5.4]{LodhaMoore} that every element of $G$ can be written in standard form.
The standard form is, however, not unique.

\section{Commutators for $G_0$}

We first study the abelianization of the group $G_0$.
The partial action of $F$ on the set of all non-constant binary sequences is transitive.
Therefore, from relation $(3)$ it follows that for all non-constant $\mathtt{s}$, $\yy{s}$ is a conjugate of $y_\mathtt{10}$.

\begin{lem}\label{lem:31}
The map $\{x,x_\mathtt{1},y_\mathtt{10}\} \to \zz^3$ given by
$$
x\mapsto(1,0,0)\qquad
x_\mathtt1\mapsto(0,1,0)\qquad
y_{\mathtt{10}}\mapsto(0,0,1)
$$
extends to a surjective homomorphism $\pi:G_0\to\zz^3$,
with kernel being the commutator subgroup $G_0'=[G_0,G_0]$.
\end{lem}

\begin{proof}
We use the infinite presentation of $G_0$ having generating set
$S_0=X\cup Y_0$  where $X=\{ \xx{s} \st \mathtt{s}\in 2^{<\nn}\}$ and
 $Y_0=\{\yy{s} \st \mathtt{s}\in 2^{<\nn}, \text{$\mathtt{s}$ is not a constant word}\}$.
The set of relations is given by those relations of the form (1) to (5) above that involve only elements of $S_0$.
That this gives a presentation of $G_0$ is Theorem 3.3 in \cite{LodhaMoore}. One could, of course, also use the finite presentation for $G_0$ on the generating set $\{x, x_\mathtt{1}, y_\mathtt{10}\}$ (see \cite{LodhaMoore}), but we shall not do so.

Noting that each element $y_\mathtt{s}\in Y_0$ is conjugate in $G_0$ to $y_{\mathtt{10}}$, that $x_{\mathtt{0}^m}\in X$ is conjugate to $x_\mathtt{0}$, that $x_{\mathtt{1}^m}\in X$
is conjugate to $x_\mathtt{1}$ and that $x_{\mathtt{s}}\in X$ is conjugate to $\xx{10}$ when $\mathtt{s}$ is non-constant, we consider the map from $S_0 \to \zz^3$ given by
$$
y_\mathtt{s}\mapsto (0,0,1)\qquad
x\mapsto(1,0,0)\qquad
x_\mathtt{s}\mapsto
\begin{cases}
(1,-1,0) &  \text{ if $\mathtt{s}=\mathtt{0}^m$ for some $m\ge 1$}\\
(0,1,0) & \text{ if $\mathtt{s}=\mathtt{1}^m$ for some $m\ge 1$}\\
(0,0,0) & \text{ if $\mathtt{s}$ is non-constant}
\end{cases}
$$
That this map extends to a homomorphism $\pi:G_0 \to \zz^3$ can be readily seen by considering the effect on each of the relations (1) to (5).

Since $\{x, x_\mathtt{1}, y_\mathtt{10}\}$ projects to a generating set for the abelianization $G_0/G_0'$ and  the image $\{\pi(x), \pi(x_\mathtt{1}), \pi(y_\mathtt{10})\}$ is a basis for $\zz^3$ we conclude that $\pi$ induces an isomorphism from $G_0/G_0'$ to $\zz^3$.
\end{proof}

The restriction of $\pi$ to the subgroup $F$ generated by $\{x,x_\mathtt{1} \}$ gives (after a change of basis) the
abelianization map for Thompson's group $F$,
where elements are evaluated by the two germs at $\pm\infty$.
The third component of the map $\pi$ represents the total exponent for the $y$-generators in a word representing an element. Hence, we have the following result.

\begin{prop}\label{prop:commutator}
The commutator $G_0'$ contains exactly  those elements in $G_0$ which have compact support, and which have a total exponent in the $y$-generators equal to zero. \qed
\end{prop}

The goal of this section is to prove the first part of Theorem \mainB, which concerns $G_0$.
To proceed with the proof, we will use the following theorem, due to Higman.
Let $\Gamma$ be a group of bijections of some set $E$. For $g\in \Gamma$ define its \emph{moved set} $D(g)$ as the set of points $x\in E$ such that $g(x)\ne x$. This is analogous to the support, but since \emph{a priori} there is no topology on $E$, we do not take the closure.

\begin{thm}\label{Higman}
Suppose that for all $\alpha, \beta, \gamma\in \Gamma\setminus \{1_{\Gamma}\}$, there is a $\rho\in \Gamma$ such that the following holds:
$\gamma(\rho(S))\cap \rho(S)=\varnothing$ where $S=D(\alpha)\cup D(\beta)$.
Then $\Gamma'$ is simple.

\end{thm}

The proof can be seen in \cite{higmansimple}.

\begin{cor}\label{HigmanCor}
Let $\Gamma$ be a group of compactly supported homeomorphisms of $\mathbb{R}$ that contains $F'$.
Then $\Gamma'$ is simple.
\end{cor}

This is the consequence of the high transitivity of $F'$ which ensures that the conditions of the theorem are satisfied.
We shall apply this corollary to several groups in the paper.

This corollary cannot be applied directly to $G_0$, since this group contains elements (e.g., $x$) whose support is all of $\rr$.
So we will apply this corollary to the group $G_0'$, whose elements have compact support.
We conclude that $G_0''$ is simple.
The proof of Theorem \mainB\ for $G_0$ will now be complete with the following result.

\begin{prop} $G_0'=G_0''$.
\end{prop}

\begin{proof}
Consider an element $g\in G_0'$ and write the element in standard form as $g=hz$, where $h\in F$ and
$$
z=y_{\mathtt{s}_1}^{a_1}\ldots y_{\mathtt{s}_n}^{a_n}
$$
for some binary sequences $\mathtt{s}_i$, and  with $a_1+\cdots+a_n=0$. Since $h$ has compact support, we have $h\in F'=F''\subset G_0''$. So we only need to show that $z\in G_0''$. Note that for any generator $\xx{s}$ such that $\mathtt{s}$ is a non constant sequence, $\xx{s}\in G_0''$. We will make use of this fact repeatedly in what follows.

The proof proceeds by induction on $k=|a_1|+\cdots+|a_n|$ which is clearly an even number. For the element $z$ there will be some $a_i$ positive and some negative. Since $G_0''$ is normal, we can cyclically conjugate  and assume that the word starts with a subword of the type $\yy{s}\yy{t}^{-1}$. As the starting point of the induction, just take $1\in G_0''$. We just need to prove that $\yy{s}\yy{t}^{-1}\in G_0''$ and using the induction hypothesis for the rest of the word, the proof will be complete. We have three cases.

Case (1): $\mathtt{s}$ and $\mathtt{t}$ are consecutive. This just means that the corresponding intervals in $\rr$ are adjacent, or that if $\mathtt{s}$ and $\mathtt{t}$ are leaves in a tree, they are consecutive in the natural order of the leaves.

Take the word $w=\yy{100}\yy{101}^{-1}\in G_0'$. Construct an element $f\in F'$ such that
$$
fwf^{-1}=\yy{10011}\yy{101}^{-1},
$$
which is possible because these two sequences are also consecutive and any of these can be $F'$-conjugated to any other. Now we have that $[w,f]=wfw^{-1}f^{-1}\in G_0''$, since it is the commutator of two elements in $G_0'$. But clearly
$$
[w,f]=\yy{100}\yy{10011}^{-1},
$$
Now apply relation $(5)$ to $\yy{100}$ to get
$$
[w,f]=x_\mathtt{100}\yy{1000}\yy{10010}^{-1}\yy{10011}\yy{10011}^{-1}
=x_\mathtt{100}\yy{1000}\yy{10010}^{-1}.
$$
As mentioned before, we know that $x_\mathtt{100}\in F'\subset F''\subset G_0''$, so this implies that $\yy{1000}\yy{10010}^{-1}\in G_0''$, and these are two consecutive binary sequences. Hence, by conjugation, any $\yy{s}\yy{t}^{-1}$ with consecutive binary sequences is in $G_0''$.

Case (2): $\mathtt{s}$ and $\mathtt{t}$ are not consecutive and also not comparable. Assume $\mathtt{s}<\mathtt{t}$, as the other case reduces to this by inverting the element. In that case, just write
$$
\yy{s}\yy{t}^{-1}=\yy{s}y_{\mathtt{s}_1}^{-1}y_{\mathtt{s}_1}y_{\mathtt{s}_2}^{-1}y_{\mathtt{s}_2}\ldots y_{\mathtt{s}_m}^{-1}y_{\mathtt{s}_m}\yy{t}^{-1}
$$
such that the pairs
$$
\yy{s}y_{\mathtt{s}_1}^{-1}\qquad y_{\mathtt{s}_i}y_{\mathtt{s}_{i+1}}^{-1}\qquad y_{\mathtt{s}_m}\yy{t}^{-1}
$$
lie in the previous case.

Case (3): $\mathtt{s}$ and $\mathtt{t}$ are comparable, so assume $\mathtt{t}=\mathtt{su}$. The case $\mathtt{s}=\mathtt{tu}$ reduces to this by taking an inverse. We apply relation $(5)$ to $\yy{s}$ again and find pairs which now correspond to cases (1) or (2). Cases $y_{\mathtt{s}_1}^{-1}y_{\mathtt{s}_2}$ are cyclically permuted to $y_{\mathtt{s}_2}y_{\mathtt{s}_1}^{-1}$. We distinguish all four easy cases for clarity:
    \begin{itemize}
\item $\mathtt{u}=\mathtt{0}$ or $\mathtt{u}=\mathtt{11}$.
Since $\yy{s}\yy{su}^{-1}=x_\mathtt{s}\yy{s0}\yy{s10}^{-1}\yy{s11}\yy{su}^{-1}$,
 $\yy{su}^{-1}$ cancels with one of the results of relation $(5)$ applied to $\yy{s}$.
 We are left with the product of an $x$-generator (in $G_0''$) with a word that falls in case (1).
\item $\mathtt{u}=\mathtt1$. Applying relation $(5)$ we obtain
$$
\xx{s}\yy{s0}\yy{s10}^{-1}\yy{s11}\yy{s1}^{-1}
$$
and we apply case (1) to the pairs
$$
\yy{s0}\yy{s1}^{-1}\qquad \yy{s10}^{-1}\yy{s11}.
$$
\item $\mathtt{u}=\mathtt{10}$. Using relations $(4),(5)$ we obtain:
$$
\xx{s}\yy{s0}\yy{s10}^{-2}\yy{s11}
$$
and we just need to apply case (1) to $\yy{s0}\yy{s10}^{-1}$ and $\yy{s10}^{-1}\yy{s11}$.
\item $\mathtt{u}\neq \mathtt0,\mathtt1, \mathtt{10},\mathtt{11}$. Using relation $(5)$ we obtain
$$
\xx{s}\yy{s0}\yy{s10}^{-1}\yy{s11}\yy{su}^{-1}.
$$
It suffices to show that $\yy{s0}\yy{s10}^{-1}\yy{s11}\yy{su}^{-1}\in G_0''$.
If $\mathtt{u}$ begins with a $\mathtt1$,
by cyclic conjugation we obtain $(\yy{su}^{-1}\yy{s0})(\yy{s10}^{-1}\yy{s11})$.
The word $\yy{su}^{-1}\yy{s0}$ falls in cases (1) or (2) and the word $\yy{s10}^{-1}\yy{s11}$ falls in (1),
so we are done.
For the case where $\mathtt{u}$ begins with a $\mathtt0$ we express the word as a product
of $\yy{s0}\yy{s10}^{-1}$ and $\yy{s11}\yy{su}^{-1}$,
both words fall in previous case.


\end{itemize}
\end{proof}

Our main theorem for $G_0$ has some important corollaries.

\begin{cor}\label{finiteindex}
 The finite-index subgroups of $G_0$ are in bijection with the finite-index subgroups of $\zz^3$. Every subgroup $H$ of finite index is normal, and $H'=G_0'$.
\end{cor}

{\it Proof.} If $H$ is not normal, then take the intersection $K$ of all its conjugates, which is now normal. Consider $K\cap G_0'$. This is a finite-index normal subgroup of $G_0'$, and since this group is simple and infinite, it has to be that $K\cap G_0'=G_0'$ and hence $G_0'\subset K\subset H$. Since every finite-index subgroup contains $G_0'$, now all of them correspond to those of the abelianization map, and hence they are all normal. The last assertion is true because $H'\subseteq G_0'$, and since $H'$ is characteristic in $H$ and hence normal in $G_0$, we have that $G_0'\subseteq H'$. \qed

For our next corollary we will need the following fact.

\begin{prop}\label{centre}
The center of $G_0$ is trivial.
\end{prop}

{\it Proof.}
Let $g\in G_0$ be in the center of $G_0$. In particular,
$g$ commutes with integer translations.

Let $I$ be an interval on which the action of $g$ is not affine.
Now the action of any piecewise projective homeomorphism near infinity is affine.
We conjugate $g$ by $x+n$ for $n\in \mathbb{N}$ to obtain a map $g^{\prime}$
which is not affine on the interval $n+I$.
By our hypothesis $g=g^{\prime}$, so it follows that $g$ is in fact piecewise affine.

Moreover, by our hypothesis it follows that the set of breakpoints of our piecewise affine
$g$ is invariant under translation, hence empty.
So $g$ is in fact an affine map.
The only affine maps that commute with integer translations are themselves translations,
so $g$ is of the form $x+t$ for $t\in \mathbb{R}$.
Now our lemma follows from the fact that $b,c$ do not commute with $x+t$
for $t\neq 0$.
\qed

We remark that the above argument is quite general. Given any group of piecewise projective homeomorphisms that contains both a translation and a non translation, then the center of the group is trivial.

So we have now the following.

\begin{cor} Every proper quotient of $G_0$ is abelian.
\end{cor}

{\it Proof.} Let $p:G_0\longrightarrow Q$ be a proper quotient map, and let $K=\ker p$. Since the quotient is not $G_0$, there exists $x\in K$ with $x\neq 1$. Since the center is trivial, then there exists $y\in G_0$ such that $[x,y]\neq 1$. But then $[x,y]\in K\cap G_0'$, which is a normal subgroup of $G_0'$,
so it follows that $G_0'\subset K$ and $Q$ is abelian.\qed

\section{Commutators for $G$}

In this section we consider the commutator subgroups $G'$ and $G''$.
Recall that $G$ is generated by the set $\{x,\xx{1},\yy{10},\yy{0},\yy{1} \}$, that
for all $n\in\nn$,  $y_{\mathtt{0}^n}$ is a conjugate of $\yy{0}$ and $y_{\mathtt{1}^n}$ is a conjugate of $\yy{1}$ and
that for all non-constant $\mathtt{s}$, $\yy{s}$ is a conjugate of $\yy{10}$.

From the relations of the form $\yy{s}=\xx{s}\yy{s0}\yy{s10}^{-1}\yy{s11}$ we see that, unlike the situation for $G_0$,
the elements $x$ and $\xx{1}$ lie in the kernel of the abelianization map.
In more detail, we have the following relation in $G$:
$$
\yy{1}=\xx{1}\yy{10}\yy{110}^{-1}\yy{111}
$$
which, combined with
$$
\yy{111}=x^{-2}\yy{1}x^{2}\qquad
\yy{110}=x^{-1}\yy{10}x
$$
shows that $\xx{1}$ is in the kernel of the abelianization.
Similarly $\xx{0}$ is in the kernel of the abelianization because of the relation:
$$
\yy{0}=\xx{0}\yy{00}\yy{010}^{-1}\yy{011},
$$
and observing that, in an analogous way as above, $\yy{00}$ is a conjugate of $\yy0$ by an element of $F$, and similarly, $\yy{010}$ and $\yy{011}$ are also conjugate by an element of $F$. Hence when we abelianize this relation and simplify, we obtain just $\xx0=1$.
That $x$ is also in the kernel then follows from the relation
$$
\xx{0}=x^2\xx{1}^{-1}x^{-1}
$$
which is a consequence of relation (1).
Since $x$ and $\xx{1}$ together generate $F$, the whole of $F$ lies in $G'$.

We obtain the following.

\begin{lem}
The map given by:
$$
\begin{array}{c}

x\mapsto (0,0,0) \qquad
\xx{1}\mapsto (0,0,0)\qquad
\yy{10}\mapsto (1,0,0)\qquad
\yy{0}\mapsto (0,1,0)\qquad
\yy{1}\mapsto (0,0,1).
\end{array}
$$
extends to a surjective homomorphism $G\to \zz^3$,
and its kernel is exactly the commutator subgroup $G'$.
\end{lem}

\begin{proof}
As in the proof of Lemma \ref{lem:31}
we consider the infinite presentation and  show that the given map extends to a homomorphism by first noting the extension to the infinite generating set.
The infinite generating set for $G$ we consider is
$S=X\cup Y$  where $X=\{ \xx{s} \st \mathtt{s}\in 2^{<\nn}\}$ and
 $Y=\{\yy{s} \st s\in 2^{<\nn}\}$.
The set of relations is given by  (1) to (5) above.
That this gives a presentation of $G$ is Thereom 3.3 in \cite{LodhaMoore}.

Consider the map from $S \to \zz^3$ given by
$$
x_\mathtt{s}\mapsto (0,0,0)\qquad
y\mapsto(-1,1,1)\qquad
y_\mathtt{s}\mapsto
\begin{cases}
(0,1,0) &  \text{ if $\mathtt{s}=\mathtt{0}^m$ for some $m\ge 1$}\\
(0,0,1) & \text{ if $\mathtt{s}=\mathtt{1}^m$ for some $m\ge 1$}\\
(1,0,0) & \text{ if $\mathtt{s}$ is non-constant}
\end{cases}
$$
That this map extends to a homomorphism $\pi:G \to \zz^3$ can be readily seen by considering the effect on each of the relations (1) to (5).

Since $\{y_\mathtt{0}, y_\mathtt{1}, y_\mathtt{10}\}$ projects to a generating set for the abelianization $G/G'$ and  the image $\{\pi(y_\mathtt{0}), \pi(y_\mathtt{1}), \pi(y_\mathtt{10})\}$ is a basis for $\zz^3$ we conclude that $\pi$ induces an isomorphism from $G/G'$ to $\zz^3$.\end{proof}


Given a word in standard form $hy_{\mathtt{s}_1}^{a_1}\ldots y_{\mathtt{s}_n}^{a_n}$, define the
\emph{left $y$-exponent sum} to be the integer given by summing the elements of
$\{a_i\st \mathtt{s}_i=\mathtt{0}^n,\ n\ge 1\}$. Similarly, define the \emph{right $y$-exponent sum} and
\emph{central $y$-exponent sum} as the sums of the sets
$\{a_i\st \mathtt{s}_i=\mathtt{1}^n,\ n\ge 1\}$ and $\{a_i\st \mathtt{s}_i \text{ is non-constant}\}$ respectively.
The above discussion established the following.

\begin{prop}
The commutator subgroup $G'$ consists precisely of those elements of $G$  that
have a standard form expression with left $y$-exponent sum, right $y$-exponent sum and central $y$-exponent sum all equal to zero.
\qed
\end{prop}

Note that elements of $G'$ need not be compactly supported (e.g., $x$).
An element of $G'$ is compactly supported precisely when $h$ is compactly supported and
all $\mathtt{s}_i$ are non-constant. Elements of $G''$ have compact support.
In fact, we have the following.

\begin{prop}
$G''=G_0'$
\end{prop}

\begin{proof}
One inclusion is clear since $G_0'=G_0''\subseteq G''$.
For the reverse inclusion recall that the action of any piecewise projective homeomorphism is affine near infinity.
The elements of $G''$ have compact support and therefore we claim they have a standard form that does not contain any elements
of the form $\yy{s}$ with $\mathtt{s}$ constant.

It is shown in \cite{Lodha} (Section $5$) that any element $g\in G$ can be represented as a normal form $g=fy_{s_1}^{t_1}...y_{s_n}^{t_n}$
with the property that it does not admit potential cancellations or potential contractions.
(These notions are defined in \cite{Lodha}, we do not recall them here since we only need a corollary of these.
Also, this is proved for the group $G_0$ but the same proof, line by line, applies to $G$.)

We use from \cite{Lodha} the notion of \emph{calculation} of a standard form on binary sequences.
This is defined for infinite sequences in Definition $3.13$ in \cite{Lodha} and the definition for finite sequences occurs
in the paragraph after Lemma $3.14$ in \cite{Lodha}.

Let $u_0$ be a finite binary sequence that contains as a prefix $s_i$ for some $1\leq i\leq n$.
Moreover, assume that for any sequence $s_j$ in the set $\{s_1,...,s_n\}$ either $s_j\subset u_0$ or $u_0\not \subseteq s_j$.
This holds for instance whenever $u_0$ is longer than all sequences in $\{s_1,...,s_n\}$.

Let $\Lambda$ be the associated calculation of $fy_{s_1}^{t_1}...y_{s_n}^{t_n}$ on $u_0$.
Since $fy_{s_1}^{t_1}...y_{s_n}^{t_n}$ does not contain potential cancellations,
$\Lambda$ does not contain potential cancellations.
So by Lemma $3.15$ in \cite{Lodha}, there is a finite binary sequence $u_1$
such that one can perform a sequence of moves on the calculation $\Lambda u_1$ to obtain a string of the form $u_2y^n$ for some $u_2\in 2^{<\mathbb{N}}$ and $n\in \mathbb{N}\setminus \{0\}$.
Here $n$ is the number of occurrences of $y^{\pm}$ in the calculation $\Lambda u_1$,
which is positive since $u_0$ contains $s_i$ as a prefix.

 The calculation of $fy_{s_1}^{t_1}...y_{s_n}^{t_n}$ on $u_0u_1 0^{2^n} 1 0^{2^n} 1...$ equals $$\Lambda u_1 0^{2^n} 1 0^{2^n} 1...$$
 By our hypothesis, upon performing moves this simplifies to
 $$u_2 y^n 0^{2^n} 1 0^{2^n} 1...$$
 and finally to $$u_2 0 1^{2^n} 0 1^{2^n} 0...$$
Hence the action of $fy_{s_1}^{t_1}...y_{s_n}^{t_n}$ on the sequence $u_0u_1 0^{2^n} 1 0^{2^n} 1...$ does not preserve tail equivalence.

In particular, the element $g$ does not preserve tail equivalence on a dense subset of $Supp(y_{s_i}^{t_i})$ for any $1\leq i\leq n$.
If $g\in G''$, it is compactly supported, and hence must preserve tail equivalence outside a compact interval.
This means that the support of each $y_{s_i}^{t_i}$ is compact.
In particular, $fy_{s_1}^{t_1}...y_{s_n}^{t_n}$ does not contain elements of the form $\yy{s}^{\pm}$ with $\mathtt{s}$ constant.

The total $y$-exponent sum of such a standard form is zero since this is true for $G'$.
It then follows from Proposition \ref{prop:commutator} that $G''\subseteq G_0'$.
\end{proof}

In particular $G''$ is simple. Note that $G''$ is strictly smaller than $G'$ since elements of $G''$ have compact support.

Recall that $x$ and $\xx{1}$ (hence any element of $F$) are in $G'$.

\begin{prop}
The group $G'/G''$ is generated by the cosets of $x$ and $\xx{1}$.
There is an isomorphism $G'/G''\to \zz^2$ given by the images of the generators:
$$
xG''\mapsto (1,0)\qquad
\xx{1}G''\mapsto (0,1).
$$
\end{prop}

\begin{proof}
Any element of $G'$ has $\yy{0}$ exponent sum equal to zero and $\yy{1}$ exponent sum equal to zero.
The relations
$
 y_{\mathtt0^{n+1}}=x^n\yy{0}x^{-n}
 $
and
$
y_{\mathtt1^{n+1}}=x^{-n}\yy{1}x^{n}
 $
 then show that
 any element of
 $G'/G''$ can be written as
 $$hy_{\mathtt{s}_1}^{a_1}\ldots y_{\mathtt{s}_n}^{a_n}G''
$$
 with $h\in F$, each $\mathtt{s}_i$ non-constant and the sum of the $a_i$ equal to zero.
 Hence, $y_{\mathtt{s}_1}^{a_1}\ldots y_{\mathtt{s}_n}^{a_n}\in G_0'$ by Proposition \ref{prop:commutator}.
 Therefore, as $G_0'=G''$, any element of $G'/G''$ can be written in the form $x^m\xx{1}^nG''$ with $m,n\in\zz$.

That  there is a homomorphism sending $xG''$ to $(1,0)$ and $\xx{1}G''$ to $(0,1)$ is clear from
the above discussion.
To see that this is surjective, note that there is a homomorphism $G'\to \zz^2$,
given by the germs at infinity.
This map must therefore be precisely $G'/G''$.
\end{proof}

We have seen that the derived series for $G$ is $G''\lhd G'\lhd G$ with
$G/G'\iso\zz^3$, $G'/G''\iso \zz^2$ and
$G''$ perfect.

\section{Commutators for $H(A)$}
In this section we consider the group $H(A)$, where $A$ is a subring of $\mathbb{R}$.
We observe a basic fact about $P_A$.
\begin{lem}\label{units}
Let $A$ be a subring of $\mathbb{R}$.
\begin{enumerate}
\item If $A$ has a unit $c\neq \pm1$,
then $\infty\in P_A$.
\item If the only units in $A$ are $\pm 1$,
then $\infty\notin P_A$.
\end{enumerate}
\end{lem}
\begin{proof}
First we consider the case where $A$ has a unit $c\neq \pm 1$.
Consider an affine map of the form $t\to c^2t$.
Then, this map fixes $0$ and $\infty$. The corresponding matrix in $PSL_2(A)$
\[ \left( \begin{array}{cc}
c & 0 \\
0 & c^{-1}  \end{array} \right)\]
is a hyperbolic matrix which fixes $\infty$.
So it follows that $\infty\in P_A$.

Now we consider the case when the only units in $A$ are $\pm1$.
Any hyperbolic matrix that fixes $\infty$
must be the the form

\[\left( \begin{array}{cc}
u & v \\
0 & u^{-1}  \end{array} \right)\]

so that $u$ is a unit that does not equal $\pm 1$.
Since there are no such units in $A$, there are no such matrices in $PSL_2(A)$,
and so $\infty\notin P_A$.
\end{proof}

The fact that $\infty$ is such an important point and that it belongs to $P_A$ only in the case where there are nontrivial units is the reason why the proof of the theorem is split in these two cases.

So we consider first the case in which $A$ has units other than $\pm1$.
We define $H_c(A)$ as the group of compactly supported elements of $H(A)$.
Since the elements of $H(A)$  are affine near infinity, it follows that $H(A)''\subseteq H_c(A)$, because two elements of the type $y=ax+b$ have a commutator of the type $y=x+k$, and two of these commute.
Due to the high transitivity of the action of $PSL_2(\mathbb{Z})$ on the real line, a variation of Corollary \ref{HigmanCor} applies to the groups $H_c(A)$ and $H(A)''$ and therefore the groups $H_c(A)'$ and $H(A)'''$ are simple.

From our hypothesis on $A$ it is clear that $H(A)'\neq H(A)''$, since all elements of $H(A)''$ are compactly supported, whereas there are maps
in $H'(A)$ that are not compactly supported.
For instance, consider a commutator of an integer translation with a map of the form $t\mapsto p^2t$, where $p\neq \pm 1$ is a unit in $A$.
Moreover, since $H(A)''\subseteq H_c(A)$ it follows that $H(A)'''$ is a normal subgroup of $H_c(A)'$
and by the simplicity of these groups we deduce that $H_c(A)'=H(A)'''$.
To establish simplicity of $H(A)''$ it suffices to show that $H(A)''=H_c(A)'$.
Indeed it suffices to show that if $g,h\in H(A)'$,
then $[g,h]\in H_c(A)'$.

Before we show this, we first build some generic elements of $H(A)$
which will be used in the proof.

\begin{defn}
Given any positive real number $r\in A$,
there is an $x\in (0,1)$ such that $\frac{x}{1-x}=x+r$,
because the graphs of $t\to \frac{t}{1-t}$ and $t\to t+r$ must intersect in $(0,1)$.
We define a map:

\[
t\cdot \gamma_r=
\begin{cases}
 t&\text{ if }t\leq 0\\
 \frac{t}{1-t}&\text{ if }0\leq t\leq x\\
 t+r&\text{ if }x\leq t\\
\end{cases}
\qquad
\]

Now the matrices associated to the maps $t\to t+r, \frac{t}{1-t}$ are
\[ \left( \begin{array}{cc}
1 & r \\
0 & 1  \end{array} \right),
\left( \begin{array}{cc}
1 & 0 \\
-1 & 1  \end{array} \right)\]
respectively.
It follows that $x$ is fixed by
\[ \left( \begin{array}{cc}
1 & r \\
0 & 1  \end{array} \right)\cdot
\left( \begin{array}{cc}
1 & 0 \\
-1 & 1  \end{array} \right)^{-1}\]
which is a hyperbolic matrix.
Therefore $x\in P_A$,
and hence $\gamma_r\in H(A)$.

Now given any $n\in \mathbb{Z}$, $r\in A, r>0$,
we define the map:

\[
t\cdot \gamma_{n,r}=
\begin{cases}
 t&\text{ if }t\leq n\\
 \frac{t-n}{1-(t-n)}+n&\text{ if }n\leq t\leq n+x\\
 t+r&\text{ if }n+x\leq t\\
\end{cases}
\qquad
\]

This map is obtained by conjugating $\gamma_r$ by
the map $t\to t+n$.
Clearly, $\gamma_{n,r}\in H(A)$.
\end{defn}

\begin{defn}
Let $r\in A$ be a negative real number.
The graph of the map $t\to t+r$ meets the map $\frac{t}{1+t}$
at some number $x$ contained in the interval $[-r, -r+1]$.
We define the map

\[
t\cdot \lambda_r=
\begin{cases}
 t&\text{ if }t\leq 0\\
 \frac{t}{1+t}&\text{ if }0\leq t\leq x\\
 t+r&\text{ if }x\leq t\\
\end{cases}
\qquad
\]

Just as in the previous definition, one checks that $x\in P_A$ and so $\lambda_r\in H(A)$.
For $n\in \mathbb{Z}$, the map $\lambda_{n,r}$ is obtained by conjugating $\lambda_r$ by $t\to t+n$.

\[
t\cdot \lambda_{n,r}=
\begin{cases}
 t&\text{ if }t\leq n\\
 \frac{t-n}{1+(t-n)}+n&\text{ if }n\leq t\leq n+x\\
 t+r&\text{ if }n+x\leq t\\
\end{cases}
\qquad
\]

It follows that $\lambda_{n,r}\in H(A)$.
\end{defn}

The idea of the construction of these elements is to provide ``bump" or ``step" functions
between the identity and $t+r$, always within $H(A)$.
Our maps in $H(A)$ are translations $t+r$ near infinity, so these functions will be used to provide steps to the identity which will transform them into compactly supported maps. See Figure \ref{gammas}.

\begin{figure}[t]
\centerline{\includegraphics [width=130mm]{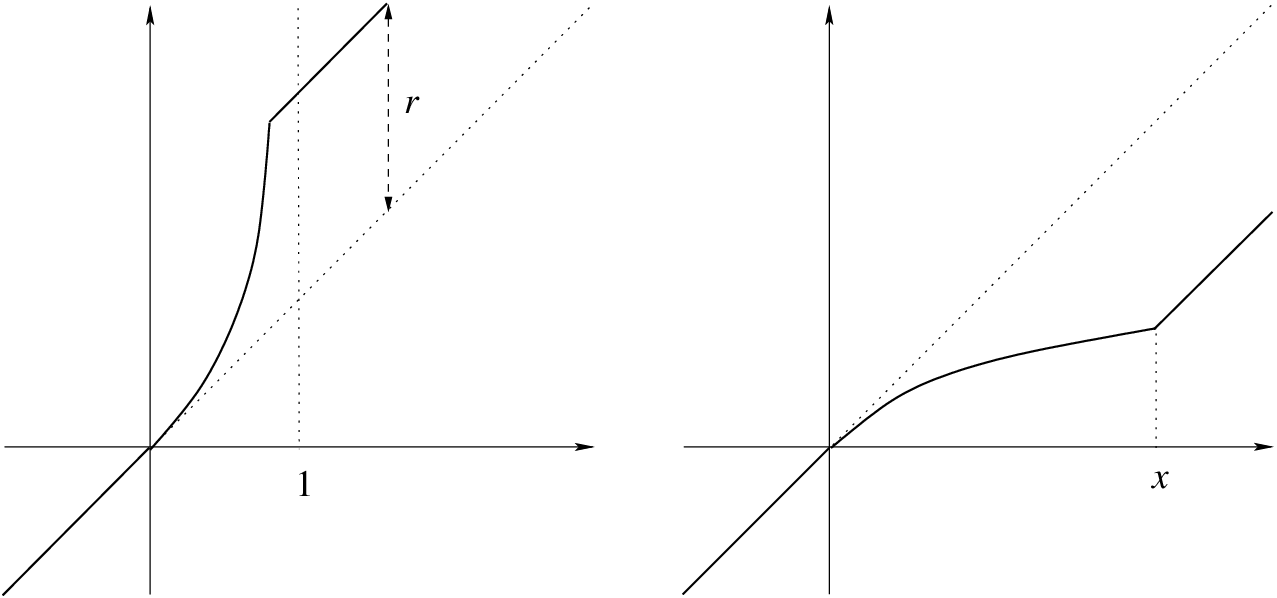}}
\caption{The maps $\gamma_r$ and $\lambda_r$.}
\label{gammas}
\end{figure}

Stated in general, the generic elements constructed above allow us to observe the following.

\begin{lem}\label{generic1}
For each $r\in A$ and  $p\in \mathbb{R}$
there is a $f\in H(A)$ such that:
\begin{enumerate}
\item $f$ is supported on $[y,\infty)$ for some $y<p$.
\item The restriction of $f$ to $(p,\infty)$ equals addition by $r$.
\end{enumerate}
\end{lem}

In an analogous fashion, one establishes the following.

\begin{lem}\label{generic2}
For each $r\in A$ and $p\in \mathbb{R}$
there is a $f\in H(A)$ such that:
\begin{enumerate}
\item $f$ is supported on $(-\infty, z]$ for some $z>p$.
\item The restriction of $f$ to $(-\infty,p)$ equals addition by $r$.
\end{enumerate}
\end{lem}

We now provide an elementary \emph{gluing construction} that allows one to build piecewise projective maps,
by gluing pieces of piecewise projective maps provided they agree on a suitable interval. If two maps agree in an interval,
we can take a hybrid of the two which consists of one on one side, and another in the other side. See Figure \ref{gluingpic}.

\begin{figure}[t]
\centerline{\includegraphics [width=130mm]{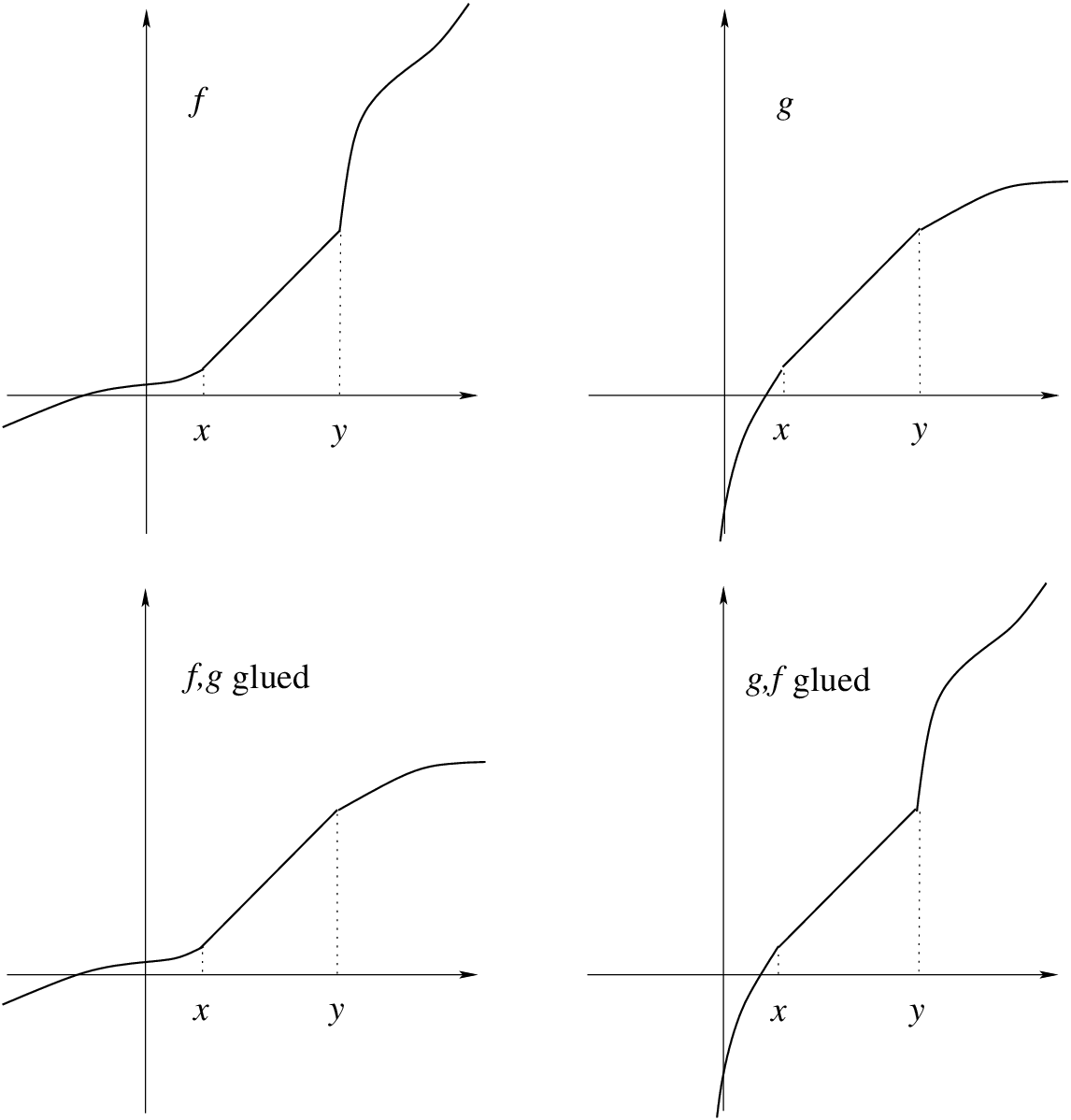}}
\caption{An example of gluing.}
\label{gluingpic}
\end{figure}

The proof of this lemma is straightforward.

\begin{lem}\label{gluing}
(Gluing) Let $f,g$ be an ordered pair of elements in $H(A)$ and let $I=[x,y]$ be such that
the restrictions of $f,g$ on $[x,y]$ agree.
Then there is an element $h\in H(A)$ such that:
\begin{enumerate}
\item  The restriction of $f,h$ on
$(-\infty, x]$ agree.
\item The restrictions of $g,h$ on $[y,\infty)$ agree.
\item The restriction of $h$ on $[x,y]$ agrees with the restrictions of both $f,g$ on $[x,y]$.
\end{enumerate}
\end{lem}

Note that in the gluing construction the order of the pair $f,g$ is essential in determining how the elements are glued, the resulting maps are different of we glue $f,g$ or if we glue $g,f$.
The interval $[x,y]$ in Lemma \ref{gluing} will be called the \emph{gluing interval}.
Now we are ready to prove the main lemma
for the case where $A$ has units other than $\pm 1$.

\begin{lem}\label{specialcomm1}
Let $g,h\in H(A)'$. Then $[g,h]\in H_c(A)'$.
\end{lem}

\begin{proof}
The main idea of the proof is to find elements $h_1,h_2, k_1,k_2\in H(A)$ such that:
\begin{enumerate}
\item $[h_1,h_2], [k_1,k_2]\in H_c(A)'$.
\item $[h_1,h_2] [g,h] [k_1,k_2]\in H_c(A)'$.
\end{enumerate}
This will finish the proof.
The construction of these elements will be done in four steps.

Step 1: The maps $f,g\in H(A)'$, they are translations near infinity.
This allows us to choose a sufficiently large interval $[r,s]$ for which
the restriction of each element of $\{f,g,f^{-1},g^{-1}\}$ to $(-\infty, r)$
and $(s,\infty)$ are translations. Outside of this interval, we will glue the step functions constructed above so our maps become compactly supported.

Step 2: Applying Lemma \ref{generic1},
we find elements $h_1,h_2$ such that:
\begin{enumerate}
\item $h_1,h_2$ are supported on an interval $[x,\infty)$ for $x<r$.
\item There is a real $x<x_1<r$ such that the restrictions of $h_1,f$
on $[x_1,r]$ agree.
\item There is a real $x<x_2<r$ such that the restrictions of $h_2,g$ on $[x_2,r]$
agree.
\item Let $j_1,j_2$ be the elements obtained by gluing $h_1,f$ and $h_2,g$ along $[x_1,r],[x_2,r]$
respectively. Then $[j_1,j_2]=[h_1,h_2][f,g]$.
\end{enumerate}
The final condition above is satisfied if the gluing intervals are sufficiently large.
Since translations commute,
as we apply the sequence of elements of the commutator $$j_1,j_2,j_1^{-1},j_2^{-1}$$
in that order, one by one, we notice that if the gluing interval is large enough, it contains a subinterval on which
each element acts like a translation, and hence the net result is the identity map.
On the right side of this piece, the commutator acts like $[f,g]$,
and on the left side it acts like $[h_1,h_2]$. See Figure \ref{h1h2}.

\begin{figure}[ph]
\vspace{20mm}
\centerline{\includegraphics [height=160mm]{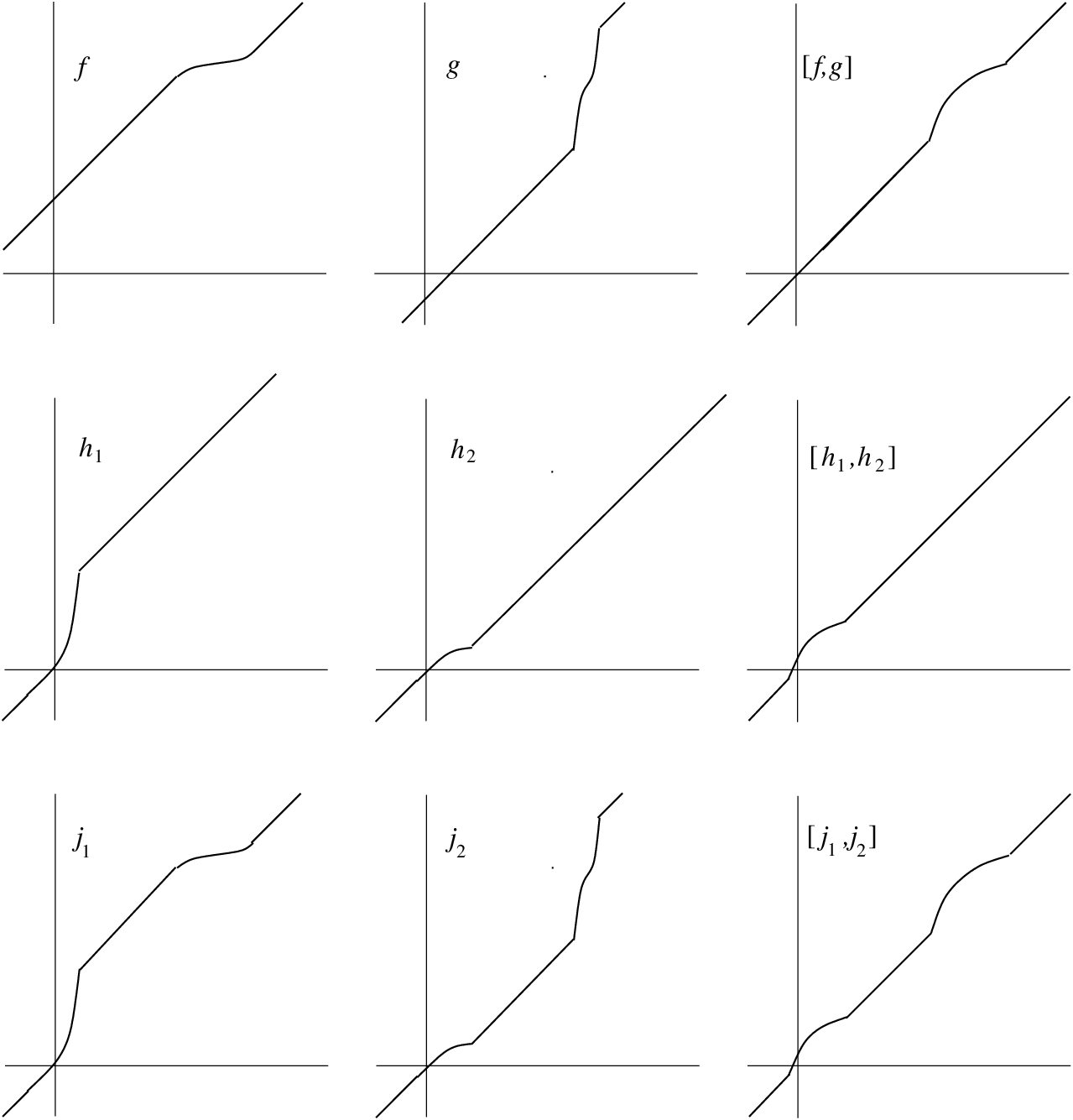}}
\caption{Step 2 in the proof that $[f,g]\in H_c(A)'$. The maps $f$ and $g$ do not have compact support, but they are glued to the maps $h_1$
and $h_2$ in such a way that the resulting maps $j_1$ and $j_2$ are the identity near $-\infty$ and their commutator $[j_1,j_2]$ agrees with $[f,g]$ except that $[h_1,h_2]$ has appeared.
The next step is to perform this procedure also near $+\infty$, to produce maps $f',g'$ which have compact support, and whose commutator agrees with $[f,g]$ except that $[h_1,h_2]$ and $[k_1,k_2]$
appear, one above, and one below $[f,g]$. The proof ends when these two latter commutators are also shown to be in $H_c(A)'$.}
\label{h1h2}
\end{figure}

Step 3: In this step we will do the same procedure as in step 2, but now on the right hand side of the maps. Applying Lemma \ref{generic2},
we find elements $k_1,k_2$ such that:
\begin{enumerate}
\item $k_1,k_2$ are supported on an interval $(-\infty,y)$.
\item There is a real $s<x_1<y$ such that the restrictions of $k_1,f$
on $[s,x_1]$ agree.
\item There is a real $s<x_2<y$ such that the restrictions of $k_2,g$ on $[s,x_2]$
agree.
\item Let $l_1,l_2$ be the elements obtained by gluing $f,k_1$ and $g,k_2$ along $[s,x_1],[s,x_2]$
respectively. Then $[l_1,l_2]=[f,g][k_1,k_2]$.
\end{enumerate}

Step 4: We glue $j_1,l_1$ along $[r,s]$
to obtain $f'$, and we glue $j_2,l_2$ along $[r,s]$ to obtain $g'$.
It follows that $[f',g']=[h_1,h_2][f,g][k_1,k_2]$,
and since $f'$ and $g'$ are constructed to be in $H_c(A)$, we conclude that $[f',g']\in H_c(A)'$.

This finishes the proof of the fact that $[h_1,h_2][f,g][k_1,k_2]\in H_c(A)'$.
To prove our lemma it suffices to show that $[h_1,h_2], [k_1,k_2]\in H_c(A)'$.
We will show this for $[h_1,h_2]$. The other case is completely analogous.
We assume that $h_1,h_2$ are supported on $[0,\infty)$ for simplicity.
(If this is not the case, we can conjugate the elements $h_1,h_2$ by a sufficiently large integer translation $p$,
and then establish that $[h_1^p, h_2^p]=[h_1,h_2]^p\in H_c'$.)

The map $t\cdot M=\frac{t}{1+t}$ fixes $0$ and maps $(0,\infty)$ to $(0,1)$.
So the map $$h_3=[M^{-1}h_1 M, M^{-1} h_2 M]=M^{-1} [h_1,h_2] M$$ is clearly in $H_c(A)'$.
In fact, the closure of the support of $h_3$ is contained in an interval $[0,t]\subset [0,1)$.

Now we will construct an element $m\in H(A)$ such that $m$ agrees with $M^{-1}$ on the support of $h_3$.
For a sufficiently large $k\in \mathbb{N}$, $\exists x\in (t,1)$ such that
$x+k=x\cdot M^{-1}=\frac{x}{1-x}$.
It follows that $x$ is fixed by
\[ \left( \begin{array}{cc}
1 & k \\
0 & 1  \end{array} \right)\cdot
\left( \begin{array}{cc}
1 & 0 \\
-1 & 1  \end{array} \right)^{-1}\]
which is a hyperbolic matrix.
So we define $m$
as:

\[
t\cdot m=
\begin{cases}
 t&\text{ if }t\leq 0\\
 \frac{t}{1-t}&\text{ if }0\leq t\leq x\\
 t+k&\text{ if }x\leq t\\
\end{cases}
\qquad
\]

The breakpoints of $m$ are $0,x,\infty$.
This means that $m\in H(A)$.
We remark that this is
the part of the argument where the existence of units in $A$ other than
$\pm 1$ is used.

However this means that $m^{-1} h_3 m= [h_1,h_2]$.
Since $H_c(A)'$ is characteristic in $H_c(A)$, it is invariant under conjugation by elements of $H(A)$.
Since $h_3\in H_c(A)'$ it follows that $[h_1,h_2]\in H_c(A)'$ as desired.
\end{proof}

We conclude the following.

\begin{cor}
$H(A)''$ is simple.
\end{cor}

The proof of Proposition \ref{centre} applies to both $H(A)$ and $H(A)'$,
so the center of these groups is trivial.
So we obtain the following.

\begin{prop}
Every proper quotient of $H(A)$ is metabelian.
\end{prop}

\begin{proof}
Let $N$ be a normal subgroup of $H(A)$.
Let $N_1=H(A)'\cap N$ and $N_2=H(A)''\cap N$.
Since $H(A)''$ is simple, either $N_2$ is trivial or $N_2=H(A)''$.
In the former case it follows that $N_1$ is in the center of $H(A)'$,
which means that $N_1$ is trivial.
This means that $N$ is in the center of $H(A)$ which means that $N$ is trivial.
So indeed it follows that either $N$ is trivial or $N$ contains $H(A)''$.
\end{proof}

This finishes the proof of Theorem \mainA\ for the case where $A$ contains units other than $\pm1$.
Now we turn our attention to the case where $A$ does not contain units other than $\pm 1$.
This condition on the units forces the elements of $H(A)$ to be translations near infinity, since the only affine maps
in $PSL_2(A)$ are translations.
It follows that
elements of $H(A)'$ are compactly supported.
By applying Higman's theorem to $H(A)'$ and $H_c(A)$
we obtain that the groups $H(A)''$ and $H_c(A)'$ are simple, and as a consequence $H_c(A)'=H(A)''$.
To prove our claim it suffices to show that $H_c(A)'=H(A)'$.

Recall that because of Lemma \ref{units}, we now have that $\infty$ is not in $P_A$.
It follows that for any $f\in H(A)$,
the translations on both germs at $\infty$ are the same.
Moreover, $\mathbb{Q}\cap P_A=\emptyset$, since the orbit of $\infty$ under the action of $PSL_2(\mathbb{Z})$ is $\mathbb{Q}\cup \{\infty\}$.

We shall need an analogue of Lemmas \ref{generic1} and \ref{generic2} in which the resulting functions satisfy that their breakpoints that lie in $\mathbb{R}$ are elements of $P_A$.

Consider the hyperbolic matrix
\[ \left( \begin{array}{cc}
-2 & 1 \\
1 & -1  \end{array} \right)\]
and the corresponding fractional linear transformation
$t\to \frac{-2t+1}{t-1}$.
The fixed points $p<q$ of this map are elements of $P_{\mathbb{Z}}\subset P_A$
and the map tends to infinity as $t\to 1$.
Moreover, $p<0<q<1$.

Given any $r\in A, r>0$,
the curves of $t\to t+r$ and $t\to \frac{-2t+1}{t-1}$
meet in a real number $s\in (q,1)$.
Moreover, $s\in P_A$ since $s$ is a fixed point of the hyperbolic matrix
\[
 \left( \begin{array}{cc}
1 & -r \\
0 & 1  \end{array} \right)
\left( \begin{array}{cc}
-2 & 1 \\
1 & -1  \end{array} \right)
=
\left( \begin{array}{cc}
-2-r & 1+r \\
1 & -1  \end{array} \right)\]

We define the map:

\[
t\cdot \zeta_{r}=
\begin{cases}
 t&\text{ if }t\leq q\\
 \frac{-2t+1}{t-1}&\text{ if }q\leq t\leq s\\
 t+r&\text{ if }s\leq t\\
\end{cases}
\qquad
\]

Upon considering inverses and conjugation of this function with integer translations, as well as the analogous functions with non-trivial germs at $-\infty$, we obtain the following analogues of Lemmas \ref{generic1} and \ref{generic2}.

\begin{lem}\label{genericexception1}
For each $r\in A$ and  $p\in \mathbb{R}$
there is a map $f$ such that:
\begin{enumerate}
\item $f$ is a piecewise $PSL_2(A)$ homeomorphism of $\mathbb{R}$.
\item $f$ is supported on $[y,\infty)$ for some $y<p$.
\item The restriction of $f$ to $(p,\infty)$ equals addition by $r$.
\item The breakpoints of $f$, besides $\infty$, all lie in $P_A$.
\end{enumerate}
\end{lem}

\begin{lem}\label{genericexception2}
For each $r\in A$ and $p\in \mathbb{R}$
there is a map $f$ such that:
\begin{enumerate}
\item $f$ is a piecewise $PSL_2(A)$ homeomorphism of $\mathbb{R}$.
\item $f$ is supported on $(-\infty, z]$ for some $z>p$.
\item The restriction of $f$ to $(-\infty,p)$ equals addition by $r$.
\item The breakpoints of $f$, besides $\infty$, all lie in $P_A$.
\end{enumerate}
\end{lem}

Now we are ready to prove the main Lemma.

\begin{lem}\label{specialcomm2}
Let $A$ be a subring of $\mathbb{R}$ whose only units are $\pm 1$.
Let $f,g\in H(A)$. Then $[f,g]\in H_c(A)'$.
\end{lem}

\begin{proof}
The proof will follow analogous lines to the proof of Lemma \ref{specialcomm1}.
We assume for the rest of the proof that the translation near infinity for $f$ is $t\to t+c_1$ and for $g$ is $t\to t+c_2$.
The main idea of the proof is to find piecewise $PSL_2(A)$ elements $h_1,h_2, k_1,k_2$ which are step functions which will transform $f$ and $g$ into compactly supported versions of themselves by stepping down the translations $t+c_i$ to the identity. Precisely, the elements $h_1,h_2, k_1,k_2$ will satisfy
\begin{enumerate}
\item $[h_1,h_2][k_1,k_2]\in H_c(A)'$
\item $[f,g][h_1,h_2] [k_1,k_2]\in H_c(A)'$
\end{enumerate}
and this will finish the proof.
We shall follow a five step procedure to construct the required elements.

Step 1: Choose a sufficiently large interval $[r,s]$ so that
the restriction of each element of $\{f,g,f^{-1},g^{-1}\}$ to $\mathbb{R}\setminus [r,s]$
is a translation. That is, all the nontranslation action for $f$ and $g$ happens well inside $[r,s]$.

Step 2: The idea now is to construct the elements $h_1,h_2$ to provide the stepping down to the identity right outside $[r,s]$, or more precisely, to the left of $r$. Applying Lemma \ref{genericexception1},
we find elements $h_1,h_2$ such that:
\begin{enumerate}
\item $h_1,h_2$ are supported on an interval $[x,\infty)$.
\item There is a real number $x_1$, with $x<x_1<r$, such that the restrictions of $h_1,f$
on $[x_1,r]$ agree.
\item There is a real number $x_2$, with $x<x_2<r$, such that the restrictions of $h_2,g$ on $[x_2,r]$
agree.
\item Let $j_1,j_2$ be the elements obtained by gluing $h_1,f$ and $h_2,g$ along $[x_1,r],[x_2,r]$
respectively. Then $[j_1,j_2]=[h_1,h_2][f,g]$.
\end{enumerate}
The last condition above is satisfied if the gluing intervals are sufficiently large,
just as in the analogous case in the proof of Lemma \ref{specialcomm1}. Observe that the elements $h_1, h_2$ are not in $H(A)$, because they have a breakpoint at $\infty$. The germ at $-\infty$ is the identity whereas the germ at $+\infty$ is a translation by $c_1$ or $c_2$.

Step 3:  Analogously to the previous step, applying Lemma \ref{genericexception2},
we find elements $k_1,k_2$ which bring the germ at $+\infty$ down to the identity. That is:
\begin{enumerate}
\item $k_1,k_2$ are supported on an interval $(-\infty,y)$.
\item There is a real number $y_1$, with $s<y_1<y$, such that the restrictions of $k_1,f$
on $[s,y_1]$ agree.
\item There is a real number $y_2$, with $s<y_2<y$, such that the restrictions of $k_2,g$ on $[s,y_2]$
agree.
\item Let $l_1,l_2$ be the elements obtained by gluing $f,k_1$ and $g,k_2$ along $[s,y_1],[s,y_2]$
respectively, then $[l_1,l_2]=[f,g][k_1,k_2]$.
\end{enumerate}

At this point we remark that by construction, the restriction of
the maps $h_1,k_1$ on $[r,s]$ equals translation by $c_1$,
and the restriction of
the maps $h_2,k_2$ on $[r,s]$ equals translation by $c_2$. All four maps will work as step functions, having a step outside $[r,s]$ to have the appropriate identity germ at $\infty$. As pointed out in step 2, these maps do not belong to $H(A)$.

Step 4: We glue $j_1,l_1$ along $[r,s]$ to obtain $s_1$,
and glue $j_2,l_2$ along $[r,s]$ to obtain $s_2$.

Step 5: We glue $h_1,k_1$ along $[r,s]$ to obtain $t_1$,
and glue $h_2,k_2$ along $[r,s]$ to obtain $t_2$.

We would like to emphasise that all four maps $s_1,s_2,t_1,t_2$ do belong to $H(A)$, and in fact, they belong to $H_c(A)$. This is because after gluing, for all four maps, both germs at $\infty$ are equal to the identity.

Finally, by construction, from the fact that the steps have been constructed far away from the nontranslation part of $f$ and $g$, it follows that $[s_1,s_2]=[h_1,h_2][f,g][k_1,k_2]$ and that this element belongs to $H_c(A)'$ because $s_1,s_2\in H_c(A)$.
Also by construction, same as before, we see that the supports of $[h_1,h_2], [f,g]$ are disjoint,
so we conclude that $$[s_1,s_2]=[h_1,h_2][f,g][k_1,k_2]=[f,g][h_1,h_2][k_1,k_2]\in H_c(A).$$
Since $$[t_1,t_2]=[h_1,h_2][k_1,k_2]\in H_c(A)'$$ because the maps $t_1,t_2$ are in $H_c(A)$, this finishes the proof.
\end{proof}

It follows from similar arguments, as in the case of $A$ with units other than $\pm 1$, that the center of $H(A)$ is trivial, and every proper quotient of $H(A)$ is abelian.
This concludes the proof of Theorem \mainA.

The groups $H(A)$ are not finitely presented, and a presentation for these groups would involve infinitely many maps similar to the map $y$ and their interactions. Writing down these generators and relations would be quite complicated. Hence this makes it difficult to compute the quotients of $H(A)$ by the commutators $H(A)'$ or $H(A)''$ to get its abelianization and metabelianization. We believe that, unlike the cases for $G$ and $G_0$, it is difficult to find easy expressions for these quotients.

\bibliographystyle{plain}
\bibliography{refs}

\end{document}